%% file: main.tex
\documentclass[12pt,twoside]{amsart}
\usepackage[all]{xy}
        \usepackage {amssymb,latexsym,amsthm,amsmath,mathtools,dsfont}
        \usepackage{enumitem,color}
        \usepackage[mathscr]{euscript}

\usepackage{mathtools}

\long\def\symbolfootnote[#1]#2{\begingroup%
\def\thefootnote{\fnsymbol{footnote}}\footnote[#1]{#2}\endgroup}

\newcommand{\Z}{\ensuremath{\mathbb{Z}}}

\makeatletter
\def\imod#1{\allowbreak\mkern10mu({\operator@font mod}\,\,#1)}
\makeatother
\makeatletter
\renewcommand*\env@matrix[1][*\c@MaxMatrixCols c]{%
  \hskip -\arraycolsep
  \let\@ifnextchar\new@ifnextchar
  \array{#1}}
\makeatother

\newtheorem{theorem}{Theorem}[section]
\newtheorem{lemma}[theorem]{Lemma}
\newtheorem{corollary}[theorem]{Corollary}

\newtheorem*{theorem*}{Theorem}
\theoremstyle{definition}
\newtheorem{definition}[theorem]{Definition}

\newtheorem{example}[theorem]{Example}
\numberwithin{equation}{section}

\newcommand{\ignore}[1]{}

\newcommand{\mynote}[1]{}
\begin{document}

\setcounter{section}{0}
\title{Idempotents in integral ring of dihedral quandle}
\author{P Narayanan and Saikat Panja}
\email{nrynnp@gmail.com, panjasaikat300@gmail.com}
\address{IISER Pune, Dr. Homi Bhabha Road, Pashan, Pune 411 008, India}
\thanks{The first named author is partially supported by the IISER Pune research fellowship and the second author has been partially supported by supported by NBHM fellowship.}
\date{\today}
\setcounter{tocdepth}{4}

\begin{abstract}
We provide the structure of some idempotents in the integral ring of a dihedral quandles, for odd case. The solution is complete in case $n=5$. This produces an example of a connected quandle $X$ such that $\textup{Aut}(X)\not\cong\textup{Aut}(\mathbb{Z}[X])$.
\end{abstract}
\maketitle
\input{intro}

\input{structure}
\input{fivecase}
\bibliographystyle{alphaurl}

\end{document}

%% file: intro.tex
\section{Introduction}
A \textit{quandle} is a pair $(A,\cdot)$ such that `$\cdot$' is a binary operation satisfying
\begin{enumerate}
	\item the map $S_a:A\longrightarrow A$, defined as $S_a(b)=b\cdot a$ is an automorphism for all $a\in A$,
	\item  for all $a\in A$, we have $S_a(a)=a$.
\end{enumerate}
The automorphisms $S_a$ are called the inner automorphisms of the quandle $A$. A quandle is said to be  \textit{connected
} if the group of inner automorphisms acts transitively on the quandle.
Let $\Z_n$ denote the cyclic group of order $n$. Then defining  $a\cdot b=2b-a$ defines a quandle structure on $A=\Z_n$. This is known as the \textit{dihedral quandle}, to be denoted by $Q_n$. Any dihedral quandle of odd order is connected. For other examples see \cite{BaPaSi19}.

\noindent To get a better understanding of the structure, a theory parallel to group rings was introduced by Bardakov, Passi and Singh in \cite{BaPaSi19}.  The quandle ring of a quandle $A$ is defined as follows. Let $R$ be a commutative ring. Consider 
\begin{displaymath}
	R[A] = \left\{\sum_{i}r_ia_i: r_i\in R,a_i\in A  \right\}.
\end{displaymath}
Then this is an additive group in usual way. Define multiplication as 
\begin{displaymath}
	\left(\sum_{i}r_ia_i\right)\cdot \left(\sum_{j}s_ja_j\right) = \sum_{i,j}r_is_j(a_i\cdot a_j).
\end{displaymath}
The terms idempotents and ring automorphism are of usual meaning, as in ring theory. Since $\Z[A]$ is an infinite ring, it is usually difficult to conclude about the idempotent structure of the ring.
Also since in quandle $a*a=a$ for all $a\in A$, these are trivial idepotents in the quandle ring $\Z[A]$.

Here we introduce the notion of an adjacency matrix of a quandle. This helps us in finding some of the idempotents in integral quandle ring of a dihedral quandle of odd order. Using this we further show that $\Z[Q_5]$ has only trivial idempotents. This further helps in finiding the automorphism group of $\Z[Q_5]$. This answers a question of Elhamdadi in negation, whether $\textup{Aut}(\Z[X])\cong\textup{Aut}(X)$, for a connected quandle $X$.

%% file: structure.tex
\section{Adjacency matrix of a quandle}
The following definition is from \cite{NeHo05}
\begin{definition}
Let $(X,*)$ be a quandle of size $n$. Define the $n\times n$ matrix $A=(a_{ij})$ with $a_{ij}=x_i*x_j$. This matrix will be
called as the \textit{adjacency matrix of the quandle}.
\end{definition}
\begin{example}
Let $Q_5$ be the dihedral quandle of size $5$. Then the adjacency matrix of $Q_5$ will be given by
\begin{align*}
    \begin{pmatrix}
    0&2&4&1&3\\
    4&1&3&0&2\\
    3&0&2&4&1\\
    2&4&1&3&0\\
    1&3&0&2&4
    \end{pmatrix}.
\end{align*}
\end{example}
\begin{lemma}\label{l001}
Let $n$ be an odd integer. Then the rows and columns of the adjacency matrix corresponding to a dihedral quandle of order $n$, is determined by two elements $\rho,\sigma\in S_n$. Here $S_n$ denotes the symmetric group on letters $\{0,1,\cdots,n-1\}$. Furthermore, the first row is determined by 
\begin{align*}
    \rho=\left(0\quad2\quad\cdots\quad n-1\quad1\quad3\quad\cdots\quad n-2\right),
\end{align*}
and the first column is determined by 
\begin{align*}
    \sigma=(0\quad1\quad2\quad\cdots\quad n-1)^{-1}.
\end{align*}
\end{lemma}
\begin{proof}
We first prove that all the rows of the adjacency matrix is determined by one row. 
To achieve this let us assume $x_i*x_j=x_k$ and $x_i*x_{j+1}=x_l$. We need to show that
$x_m*x_r=x_k$ implies $x_m*x_{r+1}=x_l$. From the given three conditions we have
\begin{align*}
    &2j-i=k,\quad 2j+2-i=l,\quad 2r-m=k\pmod{n}.
\end{align*}
Hence we have $2r+2-m=l\pmod{n}$. Now we compute the first row of the matrix. This can be easily shown to be equal to $\rho$. The comment about $\sigma$ follows similarly.
\end{proof}

\begin{lemma}\label{l002}
Let $n$ be an even integer. Then the rows and columns of the adjacency matrix corresponding to a dihedral quandle of order $n$, is determined by three elements $\rho_1,\rho_2,\sigma\in S_n$. Here $S_n$ denotes the symmetric group on letters $\{0,1,\cdots,n-1\}$. Furthermore, the row is determined by 
\begin{align*}
    \rho_1&=\left(0\quad2\quad\cdots\quad n-2\right)\\
    \rho_2&=\left(1\quad 3\quad\cdots\quad n-1\right),
\end{align*}
and the column is determined by 
\begin{align*}
    \sigma=(0\quad1\quad2\quad\cdots\quad n-1)^{-1}.
\end{align*}
\end{lemma}
\begin{proof}
Same as the last lemma and hence left to the reader.

\end{proof}
\begin{lemma}\label{l003}
Let $Q_n$ denote the dihedral quandle of odd order $n$. The $\sum\limits_{i=0}^{n-1}\alpha_ix_i\in \Z[Q_n]$ is an 
idempotent if and only if $(\alpha_0,\alpha_1,\cdots,\alpha_{n-1})$ is a solution to the system of equations given by 
\begin{align*}
    B\cdot T=T,\quad\sum\limits_{i=0}^{n-1}t_i=1,
\end{align*}
where $T=(t_0,t_1,\cdots,t_{n-1})^t$ and $b_{ij}=t_{a_{ij}}$ with $A=(a_{ij})$ being the adjacency matrix of $Q_n$.
\end{lemma}
\begin{proof}
Let $\sum_{i=0}^{n-1}\alpha_ix_i\in\Z[Q_n]$ is an idempotent element. Then we get that
\begin{align*}
    &\left(\sum_{i=0}^{n-1}\alpha_ix_i\right)^2=\sum_{i=0}^{n-1}\alpha_ix_i\\
    \implies&\sum\limits_{i=0}^{n-1}\alpha_i^2x_i+\sum\limits_{0\leq i<j\leq n}\alpha_i\alpha_j(x_i*x_j+x_j*x_i)=\sum_{i=0}^{n-1}\alpha_ix_i.
\end{align*}
Now we want to compare coefficient of $x_k$ of both side. Now note that for all $0\leq i\leq n-1$, we have the equality
\begin{align*}
    (x_k*x_i)*x_i&=x_k.
\end{align*}
Hence comparing the both side we get that
\begin{align*}
    \alpha_k=\sum_{i=0}^{n-1}\alpha_{k*i}\alpha_i.
\end{align*}
Since we know that $\sum\limits_{i=0}^{n-1}\alpha_i=1$, we get that $(\alpha_0,\alpha_1,\cdots,\alpha_{n-1})$ should satisfy the given system of equations.
A mutatis mutandis produce the other direction for the proof of the other direction.
\end{proof}
\begin{definition}
Let $a=\sum\alpha_ix_i\in\Z[X]$. We define the \textit{length of $a$} to be $\ell(a)=|\{i:\alpha_i\neq 0\}|$.
\end{definition}
\begin{lemma}
Let $n$ be odd. Then there is no idempotent $a\in\Z[Q_n]$ satisfying $\ell(a)=2$.
\end{lemma}
\begin{proof}

Suppose on the contrary assume that there exist $i\neq j$ with $\alpha_i,\alpha_j \neq 0$ and $\alpha_k = 0 $ for al $k \neq i,j$. Consider the $i$-the column of the adjacency matrix. Then as seen above it will be the permutation $\sigma$. Therefore there exist a $k$ such that $k$-th entry of the column is $\alpha_j$. Moreover $k$ cannot be equal to $i$. If $k$ were equal to $j$, then $x_j*x_i = x_k$ since $(x_j*x_i)*x_i = x_j$. This will force $i=j$, contrary to our assumptions. Hence $k \neq i, j$. So from the $k$-th row we get an equation of the form $\alpha_i.\alpha_j =0$ or $2\alpha_i.\alpha_j=0$ again contradicting our assumption $\alpha_i,\alpha_j \neq 0$.
\end{proof}

%% file: fivecase.tex
\section{Idempotents in $\Z[Q_5]$}
\begin{theorem}
The idempotents of $\Z[Q_5]$ are the trivial idempotents.
\end{theorem}
\begin{proof}
We will be using Lemma \ref{l003}. Note that the adjacency matrix is given by
\begin{align*}
    \begin{pmatrix}
    0 & 2 & 4 & 1 & 3\\
    4 & 1 & 3 & 0 & 2\\
    3 & 0 & 2 & 4 & 1\\
    2 & 4 & 1 & 3 & 0\\
    1 & 3 & 0 & 2 & 4\\
    \end{pmatrix}
\end{align*}
Hence we need to find $\Z$-solutions of the system of equations given by
\begin{align*}
    \begin{pmatrix}
    t_0 & t_2 & t_4 & t_1 & t_3\\
    t_4 & t_1 & t_3 & t_0 & t_2\\
    t_3 & t_0 & t_2 & t_4 & t_1\\
    t_2 & t_4 & t_1 & t_3 & t_0\\
    t_1 & t_3 & t_0 & t_2 & t_4\\
    \end{pmatrix}
    \begin{pmatrix}
    t_0 \\ t_1 \\ t_2 \\ t_3 \\ t_4 
    \end{pmatrix}
    =
    \begin{pmatrix}
    t_0 \\ t_1 \\ t_2 \\ t_3 \\ t_4 
    \end{pmatrix},\sum\limits_{i=0}^4t_i=1.
\end{align*}
We solve this using the Groebner basis, since a common zero of the above set 
will be a common zero of the Groebner basis. The Groebner basis of the above equations is given by
\begin{align*}
    &t_4^3-t_4^2,t_1^2-t_1,t_1t_2+3t_4^2-3t_4,t_2^2-3t_4^2-t_2+3t_4,t_1t_3-t_1^2+t_1,t_2t_3,\\
    &t_3^2-4t_2^2-t_3+4t_2,t_1t_4+t_4^2-t_4,t_2t_4,t_3t_4,5t_4^2-5t_4.
\end{align*}
This system can be easily seen to have only solutions $t_i=1$ and $t_j=0$ for all $i\neq j$, $i,j\in
\{0,1,2,3,4\}$. Hence the idempotents of $\Z[Q_5]$ are trivial.
\end{proof}
\begin{corollary}
We have that $\textup{Aut}(\Z[Q_5])\cong S_5$.
\end{corollary}
\begin{proof}
Note that an automorphism $\varphi$ of $\Z[Q_5]$ is determined by the images of $x_i$. Since $\varphi(x_i)$ is an idempotent and $\Z[Q_5]$ ahs only trivial idempotents, it follows that $\varphi(x_i)=x_j$ for some $j$. This finishes the proof.
\end{proof}